\documentclass[10pt]{article}

\usepackage{graphicx}
\usepackage{graphpap}
\usepackage[mathscr]{eucal}
\usepackage{amsmath}
\usepackage{amsfonts}
\usepackage{amssymb}
\usepackage{amsthm}
\usepackage{amstext}
\usepackage{amscd}
\usepackage{makeidx}
\usepackage{graphics}

\newtheorem{theorem}{Theorem}[section]

\newtheorem{corollary}[theorem] {Corollary}
\newtheorem{definition}[theorem]{Definition}

\newtheorem{lemma}[theorem]{Lemma}

\newtheorem{problem}[theorem]{Problem}

\begin{document}
\title{Directed Metric Dimension of\\Oriented Graphs with Cyclic Covering}
\author{Sigit Pancahayani and Rinovia Simanjuntak\\
\vspace{-1mm}
\normalsize Combinatorial Mathematics Research Group\\
\vspace{-1mm}
\normalsize Faculty of Mathematics and Natural Sciences\\
\vspace{-1mm}
\normalsize Institut Teknologi Bandung, Bandung 40132, Indonesia\\
\vspace{2mm}\scriptsize {\small {\bf e-mail}: {\tt spancahayani@gmail.com, rino@math.itb.ac.id}}}

\date{}

\maketitle

\begin{abstract}
Let $D$ be a strongly connected oriented graph with vertex-set $V$ and arc-set $A$. The distance from a vertex $u$ to another vertex $v$, $d(u,v)$ is the minimum length of oriented paths from $u$ to $v$. Suppose $B=\{b_1,b_2,b_3,...b_k\}$ is a nonempty ordered subset of $V$. The representation of a vertex $v$ with respect to $B$, $r(v|B)$, is defined as a vector $(d(v,b_1), d(v,b_2), ..., d(v,b_k))$. If any two distinct vertices $u,v$ satisfy $r(u|B)\neq r(v|B)$, then $B$ is said to be a resolving set of $D$. If the cardinality of $B$ is minimum then $B$ is said to be a basis of $D$ and the cardinality of $B$ is called the directed metric dimension of $D$.

Let $G$ be the underlying graph of $D$ admitting a $C_n$-covering. A $C_n$-simple orientation is an orientation on $G$ such that every $C_n$ in $D$ is strongly connected. This paper deals with metric dimensions of oriented wheels, oriented fans, and amalgamation of oriented cycles, all of which admitting $C_n$-simple orientations.
\end{abstract}


\section{Introduction}

\noindent Let $D=(V,A)$ be a strongly connected oriented graph with vertex-set $V$ and arc-set $A$. For $u,v \in V(D)$, the \emph{distance from $u$ to $v$}, $d(u,v)$, is the minimum length of oriented paths from $u$ to $v$. Suppose $B=\{b_1,b_2,b_3,...b_k\}$ is a nonempty ordered subset of $V$. The \emph{representation of a vertex $v$ with respect to $B$}, $r(v|B)$, is defined as a vector $(d(v,b_1), d(v,b_2), ..., d(v,b_k))$. If any two distinct vertices $u,v$ satisfy $r(u|B)\neq r(v|B)$, then $B$ is said to be a \emph{resolving set of $D$}. If the cardinality of $B$ is minimum, then $B$ is said to be a \emph{basis} of $D$ and the cardinality of $B$ is called the \emph{directed metric dimension} of $D$, $dim(D)$. This notion was introduced by Chartrand, Raines and Zhang \cite{CRZ00} as an analog to the undirected version of metric dimension introduced by Slater \cite{Sla75} and Harary and Melter \cite{HM76}. It is obvious that not every oriented graph has directed dimension, however necessary and sufficient conditions for the dimension of an oriented graph to be defined are still unknown.

\noindent Unlike the undirected version, not many results have been known on directed metric dimension. Characterization of graphs with particular directed metric dimension is only known for one-dimensional oriented graph.
\begin{theorem} \emph{\cite{CRZ00}}
Let $D$ be a nontrivial oriented graph of order $n$. Then, $dim(D)=1$ if and only if there exists a vertex $v$ in $D$ such that
\begin{description}
  \item{(i)} $D$ admits an oriented Hamiltonian path $P$ with a terminal vertex $v$ such that $id(v)=1$ and
  \item{(ii)} if $P$ in (i) is in the form $v_{n-1},v_{n-2},...,v_1,v,$ then for each pair $i,j$ of integer with $1\leq i<j\leq n-1,$ the oriented graph $D-E(P)$ does not admit the arc in the form $(v_j,v_i)$.
\end{description}
\label{dim1}
\end{theorem}
\noindent Other results include directed metric dimension of oriented trees \cite{CRZ00}, tournaments \cite{Lo13}, and some Cayley digraphs \cite{FGO06}.

\noindent In \cite{CRZ01}, Chartrand, Raines and Zhang defined a parameter called the \emph{upper orientable dimension} of a graph $G$, $ORD(G)$, which is the maximum value of $dim(D)$ among the orientations $D$ of $G$ for which $dim(D)$ is defined.

\noindent In this paper we deal with the strongly connected oriented graphs admitting a cycle covering. The graphs under consideration are oriented wheels, oriented fans, oriented friendship graphs, and amalgamation of oriented cycles.\\

\noindent A \emph{wheel}, $W_n$, is defined as $K_1+C_n$ and a \emph{fan}, $F_{m,n}$, is $\overline{K_m}+P_n$. We shall recall the (undirected) metric dimensions of the wheel $W_n$ and the fan $F_{1,n}$, which are essentially the same on general cases, studied by Buczkowski \emph{et al} \cite{BCPZ03} and Caceres \emph{et al} \cite{CHMPS05}, respectively.
\begin{theorem} \emph{\cite{BCPZ03}}
\[dim(W_n)=\left\{\begin{array}{ll}
                        3, & \hbox{$n=3,6$;} \\
                        \lfloor\frac{2n+2}{5}\rfloor, & \hbox{otherwise.}
                  \end{array}
            \right.\]
\label{Wn}
\end{theorem}
\begin{theorem} \emph{\cite{CHMPS05}}
\[dim(F_{1,n})=\left\{\begin{array}{ll}
                            1, & \hbox{for $n=1$;} \\
                            2, & \hbox{for $n=2,3$;} \\
                            3, & \hbox{for $n=6$;} \\
                            \lfloor\frac{2n+2}{5}\rfloor, & \hbox{otherwise.}
                      \end{array}
               \right.\]
\label{F1n}
\end{theorem}

\noindent For $n\geq 2$, let $\mathcal{C}=\{C_{t_i}|i=1,2,...,n\}$ be a collection of $n$ cycles. The \emph{vertex amalgamation} of cycles in $\mathcal{C}$, $Amal\{C_{t_i}\}_{i=1}^n$, is the graph constructed by joining the cycles in $\mathcal{C}$ on a common vertex called the \emph{terminal vertex}. The \emph{edge amalgamation} of cycles in $\mathcal{C}$, $Edge-Amal\{C_{t_i}\}_{i=1}^n$, is the graph constructed by joining $n$ cycles on a common edge called the \emph{terminal edge}. Iswadi \emph{et al} \cite{IBSS10} determined the dimension of $Amal\{C_{t_i}\}_{i=1}^n$ and Simanjuntak \emph{et al} \cite{SABISU} determined the dimension of $Edge-Amal\{C_{n_i}\}_{i=1}^t$ as stated bellow.
\begin{theorem} \emph{\cite{IBSS10}}
Let $Amal\{C_{t_i}\}_{i=1}^n$ be a vertex amalgamation of $n$ cycles that consists of $n_1$ number of odd cycles and $n_2$ number of even cycles. Then,
\[dim(Amal\{C_{t_i}\}_{i=1}^n)=\left\{\begin{array}{ll}
                                            n_1, & \hbox{for $n_2=0$;} \\
                                            n_1+2n_2-1, & \hbox{for $n_2>0$.}
                                      \end{array}
                               \right.\]
\label{vamal}
\end{theorem}
\begin{theorem} \emph{\cite{SABISU}}
\[t-2 \leq dim(Edge-Amal\{C_{n_i}\}_{i=1}^t) \leq t.\]
\label{eamal}
\end{theorem}
\noindent In this paper, we will also consider a generalization of vertex and edge amalgamations, which is the path amalgamation.

\section{Main Result}

\noindent We start by introducing the notion of simple orientation.
\begin{definition}
Let $G$ be a graph with cyclic covering. An orientation on $G$ is called \textbf{$C_n$-simple} if all directed $C_n$s in the resulting oriented graph are strong.
\end{definition}

\subsection{Directed Metric Dimension of Oriented Wheels}

\noindent For $n\geq 3$, a \emph{wheel}, $W_n$, is defined as $K_1+C_n$, where the vertex $c$ in $K_1$ is called the \emph{center} and the vertices $v_1, v_2, \ldots, v_n$ in $C_n$ is called the \emph{outer vertices}. First we shall characterize oriented wheels which admitting $C_3$-simple orientations.
\begin{lemma}
There exists a $C_3$-simple orientation on $W_n$ if and only if $n$ is even.
\label{simpleWn}
\end{lemma}
\begin{proof} Suppose $n$ is odd. Considering all possible orientations for the edge $cv_1$, $(c,v_1)$ or $(v_1,c)$ will result in the last $C_3$ in $W_n$ being not strong. For the sufficiency, we define the arc set $A=\{(c,v_i), (v_i,v_{i+1}), (v_i,v_{i-1})|i \ {\rm odd}\} \cup \{(v_i,c)|i \ {\rm even}\} \cup \{(v_1,v_n)\}$ which yields a $C_3$-simple orientation on $W_n$.
\end{proof}

\begin{theorem}
If $W_n$ is a wheel with even $n$ admitting a $C_3$-simple orientation then
\[dim(W_n)=\left\{\begin{array}{ll}
                        2, & \hbox{for $n=4$;} \\
                        \frac{n}{2}-1, & \hbox{for $n \geq 6$.}
                  \end{array}
            \right.\]
\label{evenWn}
\end{theorem}
\begin{proof}
There are only two possibilities of $C_3$-simple orientation on $W_n$:
\begin{description}
 \item[(A)] where $A=\{(c,v_i), (v_i,v_{i+1}), (v_i,v_{i-1})|i \ {\rm odd}\} \cup \{(v_i,c)|i \ {\rm even}\} \cup \{(v_1,v_n)\}$, or
 \item[(B)] where $A=\{(c,v_i), (v_i,v_{i+1}), (v_i,v_{i-1})|i \ {\rm even}\} \cup \{(v_i,c)|i \ {\rm odd}\} \cup \{(v_n,v_1)\}$.
\end{description}
Thus $W_n$ does not contain an oriented Hamiltonian path and so by Theorem \ref{dim1}, $dim(W_n)>1$.

\noindent For $n=4$, let $B=\{v_1,v_2\}$, then under orientation A or B, $r(c|B)=(1,2) \ {\rm or} \ (2,1)$, $r(v_3|B)=(3,1) \ {\rm or} \ (3,2)$ and $r(v_4|B)=(2,3) \ {\rm or} \ (1,3)$, respectively. Therefore, $dim(W_4)=2$.

\noindent For $n \geq 6$, based on a $C_3$-simple orientation on $W_n$, we define a partition of $V(W_n)$: $V_0=\{c\}$, $V_1=\{v_i|d(c,v_i)=1\}$, and $V_2=\{v_i|d(c,v_i)=2\}$. Since $d(x,y)=2$ for $x\in V_2$, $y\in V_1$ and $d(x,y)=3$ for $x,y\in V_2$ then we could have at most one vertex in $V_2$ omitted from a resolving set. Thus $dim(W_n) \geq \frac{n}{2}-1$. For the upper bound, let $B$ be a subset of $V_2$ of cardinality $\frac{n}{2}-1$ and $b\in B$, then $d(v,b)=1 \ {\rm or} \ 4$, for $v\in V_1$, $d(c,b)=2$, and $d(v,b)=3$, for $v\in V_2\setminus B$. Thus in $r(v|B)$, $v\in V_1$, the coordinates  related to the neighbors of $v$ are 1 and those related to the non-neighbors of $v$ are 4. Additionally, $r(c|B)=(2,2,\ldots,2)$ and $r(v|B)=(3,3,\ldots,3)$ for $v\in V_2\setminus B$. Therefore $B$ resolves $W_n$ and $dim(W_n) \leq \frac{n}{2}-1$, and we obtain $dim(W_n) = \frac{n}{2}-1$.
\end{proof}

\noindent By Lemma \ref{simpleWn}, a wheel $W_n$ with $n$ odd is not $C_3$-simple, however we could consider a subgraph of $W_n$ admitting a $C_3$-simple; the subgraph is a fan $F_{1,n}=K_1 + P_n$, where $V(P_n)=\{v_1, v_2, \ldots, v_n\}$.

%
\begin{theorem}
If $W_n$ is a wheel with odd $n$ which contains a fan $F_{1,n}$ admitting a $C_3$-simple orientation then
\[dim(W_n)=\left\{\begin{array}{ll}
                        1, & \hbox{for $n=5$;} \\
                        \frac{n-3}{2}, & \hbox{for $n \geq 7$, $od(c)>id(c)$ or}\\
                                       & \hbox{$od(c)<id(c)$ and $(v_n,v_1)\in E(W_n)$;}\\
                        \frac{n-1}{2}, & \hbox{otherwise.}
                  \end{array}
            \right.\]
\label{oddWn}
\end{theorem}
\begin{proof}
For $n=3$, it is easy to check that for all 4 possible $C_3$-simple orientations on $W_3$, we could apply Lemma \ref{dim1} to obtain $dim(W_3)=1$.

\noindent For $n=5$, it is obvious that there is no oriented Hamiltonian path in all 4 possible $C_3$-simple orientations on $W_5$, and by Lemma \ref{dim1}, $dim(W_5)>1$. Based on a $C_3$-simple orientation on $W_n$, we define a partition of $V(W_n)$: $V_0=\{c\}$, $V_1=\{v_i|d(c,v_i)=1\}$, and $V_2=\{v_i|d(c,v_i)=2\}$. Let $B$ be a $2$-subset of $V_2$ and $b\in B$, then in $r(v|B)$, $v\in V_1$, the coordinates related to the neighbors of $v$ are 1 and those related to the non-neighbors of $v$ are 4. Additionally, $r(c|B)=(2,2)$ and $r(v|B)=(3,3)$ for $v\in V_2\setminus B$. Therefore $dim(W_n) \leq 2$.

\noindent For $n \geq 7$, as in $n=5$ before, we define a partition $V_0, V_1, V_2$ of $V(W_n)$. We shall consider two cases separately. (i) If $od(c)>id(c)$ or $od(c)<id(c)$ and $(v_1,v_n)\in E(W_n)$ then $d(x,y)=2$ for $x\in V_2$, $y\in V_1$ and $d(x,y)=3$ for $x,y\in V_2$. Thus we could have at most one vertex in $V_2$ omitted from a resolving set. Thus $dim(W_n) \geq |V_2|-1$. (ii) If $od(c)<id(c)$ and $(v_n,v_1)\in E(W_n)$ then the distances will be the same as in case (i) except for $d(v_n,v_1)$, which obviously is 1. Thus at most one vertex in $V_2$ along with $v_n$ could be omitted from a resolving set and so $dim(W_n) \geq |V_2|-2$. For the upper bound, we shall the two afore-mentioned cases. (i) If $od(c)>id(c)$ or $od(c)<id(c)$ and $(v_1,v_n)\in E(W_n)$ then let $B$ be a subset of $V_2$ of cardinality $|V_2|-1$ and $b\in B$. Thus $d(v,b)=1 \ {\rm or} \ 4$, for $v\in V_1$, $d(c,b)=2$, and $d(v,b)=3$, for $v\in V_2\setminus B$. Therefore in $r(v|B)$, $v\in V_1$, the coordinates related to the neighbors of $v$ are 1 and those related to the non-neighbors of $v$ are 4. Additionally, $r(c|B)=(2,2,\ldots,2)$ and $r(v|B)=(3,3,\ldots,3)$ for $v\in V_2\setminus B$. In other words, $B$ resolves $W_n$ and $dim(W_n) \leq |V_2|-1$. (ii) If $od(c)<id(c)$ and $(v_n,v_1)\in E(W_n)$ then let $B$ be a subset of $V_2$ of cardinality $|V_2|-2$, where $v_n \notin B$, and $b\in B$. We then obtain the same representations as in case (i) except for $r(v_n|B)$ where the coordinate related to $v_1$ is 1 and the coordinates related to the other vertices are 3. Thus $B$ resolves $W_n$ and $dim(W_n) \leq |V_2|-2$. This completes the proof.
\end{proof}

\subsection{Directed Metric Dimension of Oriented Fans}

A \emph{fan}, $F_{m,n}$, is defined as $\overline{K_m}+P_n$, where the vertices $c_1,c_2,\ldots,c_m$ in $\overline{K_m}$ is called the \emph{centers}.

\begin{theorem}
If $F_{m,n}$ is a fan admitting $C_3-$simple orientation then
\[dim(F_{m,n})=\left\{\begin{array}{ll}
                            1 & \hbox{for $m=1$ and $n=2,3,4$}; \\
                            m-1, & \hbox{for $m\geq 2$ and $n=2$}; \\
                            m, & \hbox{for $m\geq 2$ and $n=3,4$}; \\
                            m+1, & \hbox{for $m\geq 2$ and $n=5$}; \\
                            \frac{n}{2}+m-2, & \hbox{for $n$ even, $n\geq 6$}; \\
                            \frac{n-1}{2}+m-2, & \hbox{for $n$ odd, $n\geq 7$, and $od(c_i)>id(c_i), \forall i$}; \\
                            \frac{n-1}{2}+m-1, & \hbox{for $n$ odd, $n\geq 7$, and $od(c_i)<id(c_i), \forall i$}.
                      \end{array}
               \right.\]
\end{theorem}
\begin{proof} For $m=1$ and $n=2,3,4$, all $F_{m,n}$s have a directed Hamiltonian path which satisfies the premises of Theorem \ref{dim1}. (See Figure \ref{Gam3.5}).
      \begin{figure}[h!]
      \begin{center}
      \includegraphics[width=11cm]{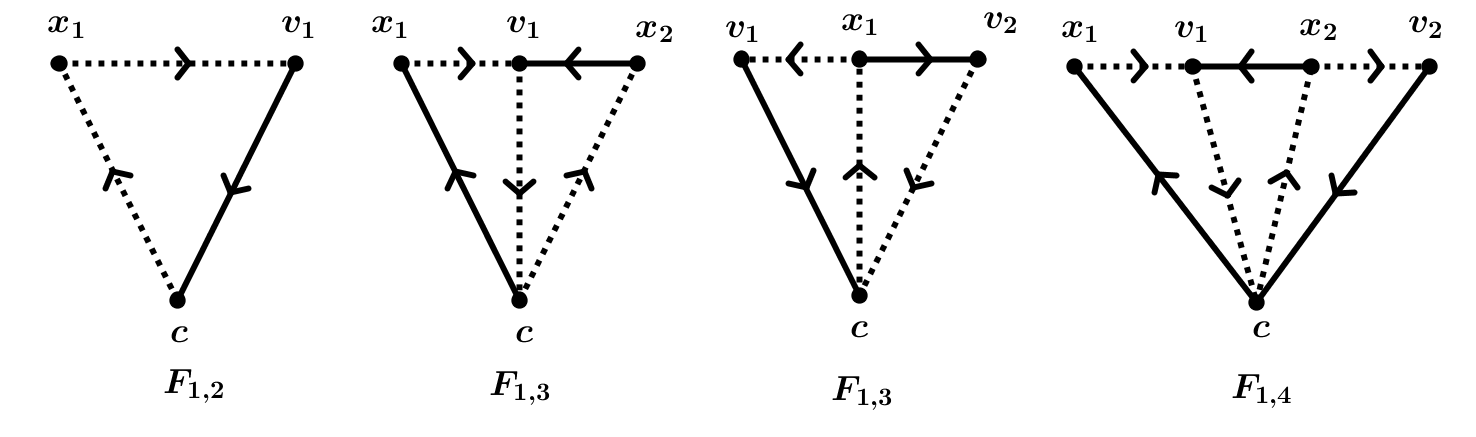}\\
      \caption{All possible $C_3$-simple orientations of the fan $F_{1,n}, n=2,3,4$. (A Hamiltonian path is depicted in dotted line.)}
      \label{Gam3.5}
      \end{center}
      \end{figure}

\noindent For $m\geq 2$ and $n=2$, we have $d(c_i,v)=d(c_j,v)$ for $v\in P_2$ and $d(c_i,c_j)=3 \ \forall i\neq j$. Thus at most one center vertex could be omitted from a resolving set and so $dim(F_{m,n})\geq m-1$. With $B=\{c_i|1\leq i\leq m-1\}$, it is easy to see that each vertex has distinct representation with respect to $B$, and so $dim(F_{m,2})\leq m-1$.

\noindent For the rest of the cases, based on a $C_3$-simple orientation on $F_{m,n}$, we define a partition of $V(F_{m,n})$: $V_0=\{c_1,c_2,\ldots,c_m\}$, $V_1=\{v_i|d(c,v_i)=1\}$, and $V_2=\{v_i|d(c,v_i)=2\}$. For $i=1,2, \ldots, m$, $d(c_i,v)=d(c_j,v)$ for $v\in V_1 \cup V_2$ and $d(c_i,c_j)=3 \ \forall i\neq j$; thus at most one center vertex could be omitted from a resolving set. Moreover, we have $d(x,y)=2$ for $x\in V_2$, $y\in V_1$ and $d(x,y)=3$ for $x,y\in V_2$, and so we could have at most one vertex in $V_2$ omitted from a resolving set. Thus $dim(F_{m,n}) \geq (m-1) + (|V_2|-1) = m+|V_2|-2$. For the upper bound, let $B_1$ be a subset of $V_0$ of cardinality $m-1$ and $B_2$ be a subset of $V_2$ of cardinality $|V_2|-1$. For $b_1\in B_1$, we have $d(c_i,b_1)=3$, for $c_i\in V_0\setminus B_1$,  $d(v,b_1)=2$, for $v\in V_1$, and $d(v,b_1)=1$, $v\in V_2$. For $b_2\in B_2$, we have $d(c_i,b_2)=2$, $d(v,b_2)=1 \ {\rm or} \ 4$, for $v\in V_1$, and $d(v,b_2)=3$, for $v\in V_2\setminus B_2$. Let $B=B_1 \cup B_2$, then in $r(v|B)$, $v\in V_1$, the coordinates related to the centers are 2, to the neighbors of $v$ are 1, and to the non-neighbors of $v$ are 4. Additionally, $r(c_i|B)=(3,3,\ldots,3,2,2,\ldots,2)$ for $c_i\in V_0\setminus B_1$ and $r(v|B)=(1,1,\ldots,1,3,3,\ldots,3)$ for $v\in V_2\setminus B$. Therefore $B$ resolves $F_{m,n}$ and $dim(F_{m,n}) \leq m+|V_2|-2$, and we obtain $dim(F_{m,n}) = m+|V_2|-2$. Depends on parity of $n$ and degrees of the centers, $|V_2|$ is either $\frac{n}{2}$, or $\frac{n-1}{2}$, or $\frac{n-1}{2}+1$. Substituting $|V_2|$ with appropriate number proves the theorem.
\end{proof}

\subsection{2-Dimensional Oriented Wheels and Fans}

Here we shall construct an orientation on $W_n$ and $F_{1,n}$ in such a way that the directed metric dimension is 2.

\begin{theorem}
For $n\geq 3$, there exists an orientation on the wheel $W_n$ such that $dim(W_n)=2$.
\end{theorem}
\begin{proof} For $3 \leq n \leq 7$ we have a $C_3$-simple orientation for $W_n$ such that $dim(W_n)=2$ as in Theorem \ref{evenWn} and \ref{oddWn}. For $n\geq 8$, consider a subgraph $F_{1,7}$ with vertex-set $\{c,v_1,v_2,\ldots,v_7\}$. Apply a simple$-C_3$ orientation with $od(c)>id(c)$ on the $F_{1,7}$. We then apply an orientation such that there exists an oriented path from $v_1$ to $v_7$, and lastly, for $v\in V(W_n)\setminus V(F_{1,7})$, set an arc from $c$ to $v$.

\noindent The next step is to determine the directed dimension of $W_n$ under such an orientation. Choose $B=\{v_2,v_4\}$. Then the representations of all vertices are $r(c|B)=(2,2), r(v_1|B)=(1,4), r(v_2|B)=(0,3), r(v_3|B)=(1,1), r(v_4|B)=(3,0), r(v_5|B)=(4,1), r(v_6|B)=(3,3), r(v_7|B)=(4,4)$, and $r(v_i|B)=(i-7+4,i-7+4)$, for $8\leq i\leq n$. Since $W_n$ does not admit an oriented Hamiltonian path then $dim(W_n)=2$.
\end{proof}

\begin{theorem}
For $n\geq 3$, there exists an orientation on the fan $F_{m,n}$ such that $dim(F_{m,n})=2$.
\end{theorem}
\begin{proof} For $n=3$, let the arc set of $F_{1,3}$ be $A=\{(v_1,v_2), (v_3,v_2), (v_1,c),$ $(v_2,c), (c,v_3)\}$. It is clear that $B=\{v_2,v_3\}$ is a resolving set. Since such orientation does not admit an oriented Hamiltonian path, we have $dim(F_{1,3})=2$. For $n=4$, let the arc set of $F_{1,4}$ be $A=\{(v_1,v_2),(v_3,v_2),$ $(v_3,v_4),(v_1,c),(v_2,c),(c,v_3),(v_4,c)\}$. Thus we have $B=\{v_2,v_3\}$ as a resolving set and, since $F_{1,4}$ does not admit an oriented Hamiltonian path, we have $dim(F_{1,4})=2$.

\noindent for $n\geq 5$, consider a subgraph $F_{1,4}$ with vertex-set $\{c,v_1,v_2,v_3,v_4\}$. Give $F_{1,4}$ an orientation as in the previous case. Then apply an orientation such that there exists an oriented path from $v_n$ to $v_4$ and, lastly, for $v\in V(W_n)\setminus V(F_{1,7})$, set an arc from $u$ to $v$. Now, let $B=\{v_2,v_3\}$. The representation of all vertices are $r(c|B)=(2,1), r(v_1|B)=(1,2), r(v_2|B)=(0,2), r(v_3|B)=(0,1), r(v_4|B)=(3,2)$, and $r(v_i|B)=(i-4+3,i-4+2)$, for $5 \leq i\leq n$. Since $F_{1,n}$ does not admit the directed Hamiltonian path, then $dim(F_{1,n})=2$.
\end{proof}

\begin{figure}[h!]
\begin{center}
  \includegraphics[width=11cm]{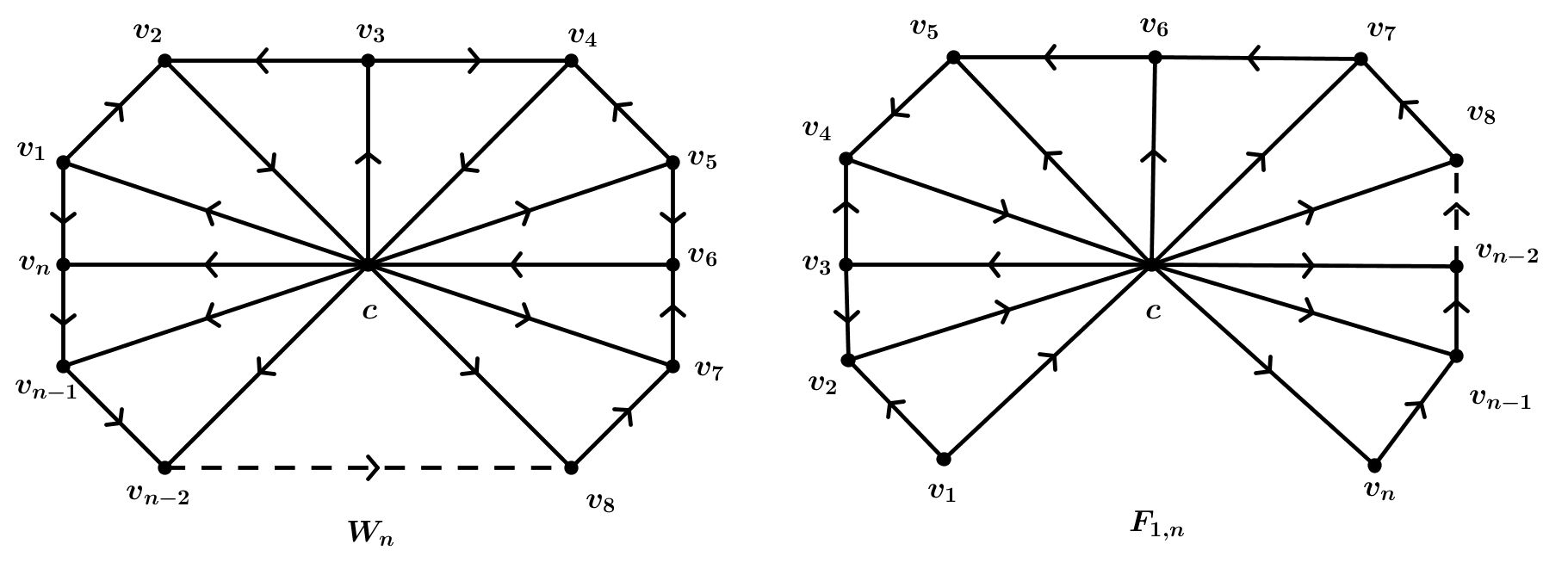}\\
  \caption{The 2-dimensional oriented $W_n$ and $F_{1,n}$.}
  \label{Gam3.14}
\end{center}
\end{figure}

\noindent So far we have obtained directed dimension for wheels and fans admitting $C_3$-simple orientations and construct $2$-dimensional oriented wheels and fans. A question arisen after this study is whether the $C_3$-simple orientation gives the maximum possible dimension, i.e. the $ORD(W_n)$ and $ORD(F_n)$, or is there an orientation which provides larger dimension? Another interesting question is whether all values up to the $ORD(W_n)$ or $ORD(F_n)$ are achievable.

\begin{problem}
For $n\geq 3$, determine $ORD(W_n)$ and $ORD(F_{m,n})$.
\label{ord}
\end{problem}
\begin{problem}
Let $k_W$ be an integer in $[2,ORD(W_n)]$ and $k_F$ be an integer in $[2,ORD(F_{m,n})]$.
Does there exist an orientation on $W_n$ such that $dim(W_n)=k_W$? Similarly, does there exist an orientation on $F_{m,n}$ such that $dim(F_{m,n})=k_F$?
\label{allvalues}
\end{problem}

\subsection{Directed Metric Dimension of Amalgamation of Oriented Cycles}

Let $\mathcal{C}=\{C_{n_i}|i=1,2,...,t\}$ be a collection of $t$ strongly oriented cycles. A \emph{path amalgamation} of cycles in $\mathcal{C}$, denoted by $P_x-Amal\{C_{n_i}\}_{i=1}^t$, where $2 \leq x \leq \min \{|V(C_{n_i})|\}-1$, is the graph constructed by joining the cycles in $\mathcal{C}$ on a
common path or order $x$ called the \emph{terminal path}. We shall denote the vertices is $P_x-Amal\{C_{t_i}\}_{i=1}^n$ as follow: $P_x=v_1v_2\ldots v_{x-1}v_x$ and $C_{n_i}=v_1v_2\ldots v_xv_{x+1}^iv_{x+2}^i\ldots v_{n_i-1}^iv_{n_i}^iv_1$. We shall use the notation $P_{n_i}$ for the path $v_{x+1}^iv_{x+2}^i\ldots v_{n_i-1}^iv_{n_i}$.

%
%
\begin{figure}[!h]
\begin{center}
  \includegraphics[width=8cm]{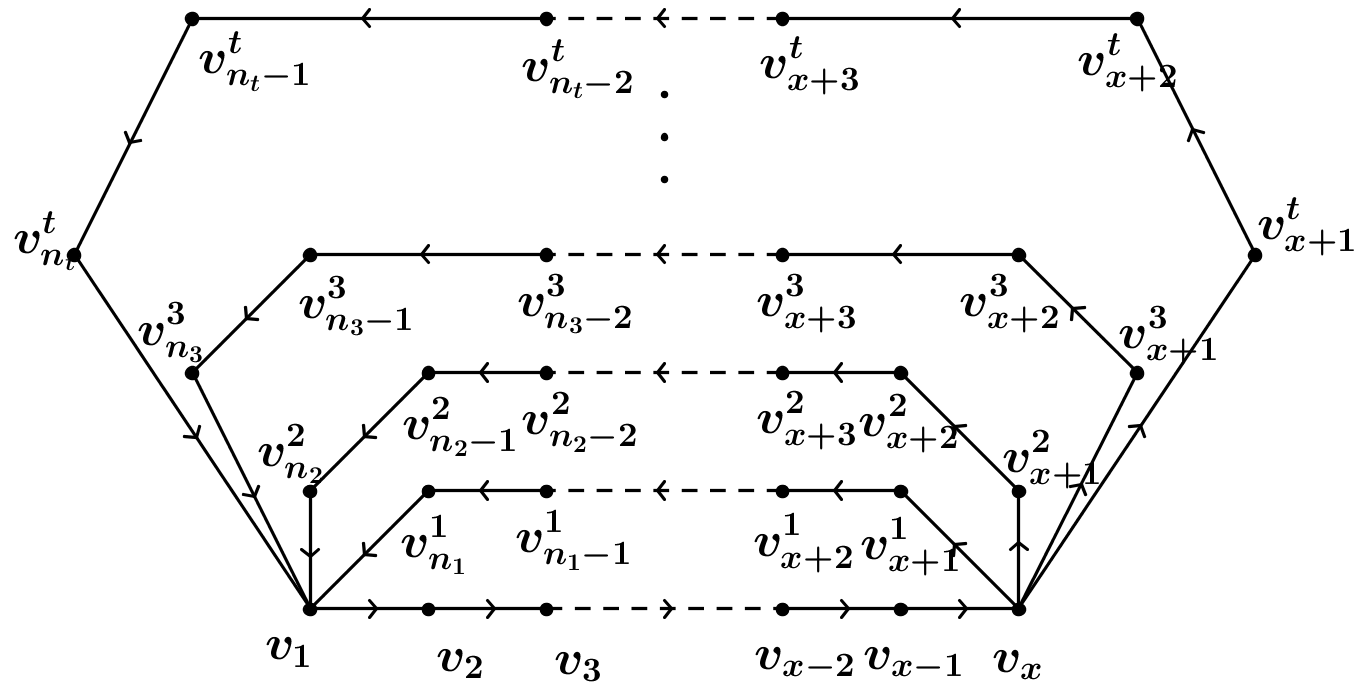}\\
  \caption{A $P_x-Amal\{C_{n_i}\}_{i=1}^t$ of strongly oriented cycles.}
  \label{Gam4.3}
\end{center}
\end{figure}

\begin{theorem}
Let $P_x-Amal\{C_{n_i}\}_{i=1}^t$ be a path amalgamation of $t$ strongly oriented cycles then $dim(P_x-Amal\{C_{n_i}\}_{i=1}^t)=t-1.$
\end{theorem}
\begin{proof} If an arbitrary set $S$ does not contain a vertex in $P_{n_i}$ and $P_{n_j}$ where $i\neq j$, then $r(v_{n_i}^i|S)=r(v_{n_j}^j|S)$ and so $S$ is not a resolving set. Thus each resolving set needs to contain a vertex from each pair of $P_{n_i}$ and $P_{n_j}$. This results in $dim(P_x-Amal\{C_{n_i}\}_{i=1}^t)\geq t-1$.

\noindent Next, consider $B=\{v_{n_1}^1,v_{n_2}^2,...,v_{n_{t-1}}^{t-1}\}$. For a vertex $u\in P_{n_i}$, we have $d(u,v_{n_i}^i)\leq n_i -1$, while for a vertex $u\in P_{n_j}, j\neq i$, $d(u,v_{n_i}^i)=d(u,v_1)+(x-1)+(n_i-x)=d(u,v_1)+n_i-1$ and for a vertex $u\in P_x$, $d(u,v_{n_i}^i)=d(u,v_x)+n_i-x$. This guarantees that each vertex will have distinct representation with respect to $B$. Thus $B$ is a resolving set and $dim(P_x-Amal\{C_{n_i}\}_{i=1}^t)\leq t-1$. This completes the proof.
\end{proof}

\noindent If the path $P_x$ is of order 1 then we have a vertex amalgamation of strongly oriented cycles and if it is of order 2 then we have an edge amalgamation of strongly oriented cycles. Thus the following corollaries hold.
\begin{corollary}
Let $Amal\{C_{n_i}\}_{i=1}^t$ be a vertex amalgamation of $t$ strongly oriented cycles then $dim(Amal\{C_{n_i}\}_{i=1}^t)=t-1.$
\end{corollary}

\begin{corollary}
Let $Edge-Amal\{C_{n_i}\}_{i=1}^t$ be an edge amalgamation of $t$ strongly oriented cycles then $dim(Edge-Amal\{C_{n_i}\}_{i=1}^t)=t-1.$
\end{corollary}

\end{document}